\documentclass[oneside,a4paper,11pt,notitlepage]{article}
\usepackage[left=0.8in, right=0.8in, top=1.5in, bottom=1.5in]{geometry}
\usepackage{pifont}
\usepackage{makecell}
\usepackage[T1]{fontenc} 
\usepackage[utf8]{inputenc} 
\usepackage[english]{babel}
\usepackage{lipsum} 
\usepackage{lmodern}
\usepackage{amssymb}
\usepackage{amsthm}
\usepackage{bm}
\usepackage{mathtools}
\usepackage{braket}
\usepackage{esint}
\usepackage{lineno}
\newcommand{\abs}[1]{{\left|#1\right|}}

\usepackage{tabularx}
\usepackage{fontawesome}
\usepackage{booktabs}
\usepackage{graphicx}
\usepackage{tikz}
\usetikzlibrary{patterns}
\usepackage{multicol}
\usepackage{caption}
\usepackage{enumerate}
\usepackage{multicol}
\usepackage[skins,theorems]{tcolorbox}
\tcbset{highlight math style={enhanced,
		colframe=black,colback=white,arc=0pt,boxrule=1pt}}
\def\XXint#1#2#3{{\setbox0=\hbox{$#1{#2#3}{\int}$}
    \vcenter{\hbox{$#2#3$}}\kern-.5\wd0}}

\theoremstyle{definition}
\newtheorem{definizione}{Definition}[section]
\theoremstyle{plain}
\newtheorem{teorema}{Theorem}[section]

\newtheorem{lemma}[teorema]{Lemma}
\newtheorem{prop}[teorema]{Proposition}
\newtheorem{corollario}[teorema]{Corollary}
\theoremstyle{definition}
\newtheorem{esempio}{Example}[section]
\newtheorem{oss}[esempio]{Remark}

\DeclareMathOperator{\R}{\mathbb{R}}

\DeclareMathOperator{\diam}{\, \textup{diam}}

\makeatletter
\newcommand{\myfootnote}[2]{\begingroup
	\def\@makefnmark{}%
	\addtocounter{footnote}{-1}%
	\footnote{\textbf{#1} #2}
	\endgroup}
\makeatother

\usepackage{soul}

\definecolor{OliveGreen}{rgb}{0,0.6,0}

\usepackage{hyperref}
\hypersetup{linktoc=none, bookmarksnumbered, colorlinks=true, linkcolor=red}

 \title{On the optimal sets in P\'olya and Makai type inequalities}
\author{Vincenzo Amato, Nunzia Gavitone, Rossano Sannipoli}
\date{}

\newcommand{\Addresses}{{
  \bigskip 
   \footnotesize 
 \noindent \textit{E-mail address}, V.~ Amato: \texttt{v.amato@ssmeridionale.it} 
  
   \medskip 
 
  \noindent\textsc{Mathematical and Physical Sciences for Advanced Materials and Technologies, Scuola Superiore Meridionale, Largo San Marcellino 10, 80138 Napoli, Italy. }

 \medskip

   \textit{E-mail address}, N.~Gavitone: \texttt{nunzia.gavitone@unina.it} 
   
 \medskip
  \textsc{Dipartimento di Matematica e Applicazioni ``R. Caccioppoli'', Universit\`a degli studi di Napoli Federico II, Via Cintia, Complesso Universitario Monte S. Angelo, 80126 Napoli, Italy.}
  \medskip
 
  \textit{E-mail address}, R.~Sannipoli: \texttt{rossano.sannipoli@unipd.it} 
  
     \medskip 
\textsc{Dipartimento di Matematica “Tullio Levi-Civita”, Universit\'a degli Studi di Padova, Via Trieste 63, 35131 Padua, Italy.}\par\nopagebreak 

}} 
\makeatletter
\def\Cline#1#2{\@Cline#1#2\@nil}
\def\@Cline#1-#2#3\@nil{%
  \omit
  \@multicnt#1%
  \advance\@multispan\m@ne
  \ifnum\@multicnt=\@ne\@firstofone{&\omit}\fi
  \@multicnt#2%
  \advance\@multicnt-#1%
  \advance\@multispan\@ne
  \leaders\hrule\@height#3\hfill
  \cr}
\makeatother

\definecolor{verde}{RGB}{20,150,100}
\definecolor{purple}{RGB}{200,30,200}

\begin{document}

\maketitle

\begin{abstract} 
  In this paper, we examine some shape functionals, introduced by P\'olya and Makai, involving the torsional rigidity and the first Dirichlet-Laplacian eigenvalue for bounded, open and convex sets of $\R^n$. We establish quantitative bounds, which give us key properties and information on the behavior of the optimizing sequences.
In particular, we consider two kinds of remainder terms that provide information about the structure of these minimizing sequences, such as information about the thickness.  

\end{abstract}

\section{Introduction}
Let $\Omega\subset\mathbb{R}^n$, $n\geq 2$, be a non-empty, bounded, open and convex set. This paper deals with shape functionals involving two well known quantitites, namely the torsional rigidity, denoted by $T(\Omega)$, and the first Dirichlet eigenvalue of the Laplacian $\lambda(\Omega)$. Their variational characterizations are given by
$$
T(\Omega) = \max_{\substack{\varphi \in H^1_0(\Omega) \\ \varphi \not\equiv 0}} \frac{\left(\displaystyle{\int_\Omega \varphi \, dx }\right)^2}{\displaystyle{\int_\Omega \abs{\nabla \varphi}^2 \, dx }} \qquad \text{ and }\qquad\lambda(\Omega)= \min_{\substack{\varphi \in H^1_0(\Omega) \\ \varphi \not\equiv 0}} \frac{\displaystyle{\int_\Omega  \abs{\nabla \varphi}^2 \, dx }}{\displaystyle{\int_\Omega  \varphi^2 \, dx }} .
$$
These functionals are respectively monotonically increasing and decreasing with respect to the set inclusion, and satisfy the following scaling properties for all $t>0$
\begin{equation*}
T(t\Omega)= t^{n+2}T(\Omega),\qquad \lambda(t\Omega)= t^{-2}\lambda(\Omega).
\end{equation*}
Concerning shape optimization issue, there are two acclaimed inequalities for which the ball is the optimum when a measure constraint is imposed. Let $\Omega$ be any open set in $\mathbb R^n$ with finite Lebesgue measure and $B$ any ball. Then, the first one is the \textit{de Saint-Venant inequality}, conjectured in \cite{stvenant} and proved in \cite{PTR48}, stated in the following scaling invariant way
\begin{equation*}
    \abs{\Omega}^{-\frac{n+2}{n}}T(\Omega) \le \abs{B}^{-\frac{n+2}{n}}T(B),
\end{equation*}
where $\abs{\Omega}$ is the Lebesgue measure of $\Omega$. The second one is the \textit{Faber-Krahn inequality}, see  \cite{Faber1923,Krahn1925}, for which we have
\begin{equation*}
    \abs{\Omega}^\frac{2}{n}\lambda(\Omega) \ge \abs{B}^\frac{2}{n}\lambda(B). 
\end{equation*}
Moreover, different inequalities involving  both quantities have been established starting from the second half of the $20$th century (see for instance \cite{KJ1,KJ2,PS}).\\
In this paper, we focus on the following shape functionals:

\begin{equation}
\label{eq:functionals}
    \begin{aligned}
    &\textbf{(i)}\;\;\displaystyle{ \frac{T(\Omega) P^2(\Omega)}{\abs{\Omega}^3}}  & \textit{(P\'olya torsion  functional)},\\ & \textbf{(ii)}\;\;\frac{\lambda( \Omega) \abs{\Omega}^2}{P^2(\Omega)}  & \textit{ (P\'olya eigenvalue  functional)}, \\
    &\textbf{(iii)}\;\;  \frac{T(\Omega)}{R_{\Omega}^2\abs{\Omega}} & \textit{(Makai functional)},  \\
    &\textbf{(iv)}\;\;\lambda( \Omega)R^2_\Omega & \textit{(Hersch functional)},
    \end{aligned}
\end{equation}
where $P(\Omega)$ and $R_\Omega$ denote the perimeter and  inradius of $\Omega$,  respectively (see Section \ref{section_notion} for the precise definitions).  


Starting with the P\'olya torsion  functional, we recall that in \cite{makai, polya1960}, Makai and P\'olya respectively proved, in the planar case, that the functional in $\textbf{(i)}$ is bounded from both above and below in the class of convex sets:
\begin{equation}
\label{polyatorsion}
\frac{1}{3} \leq  \frac{T(\Omega) P^2(\Omega)}{\abs{\Omega}^3} \leq \frac{2}{3},
\end{equation}
and showed that the inequalities are sharp, in the sense that the lower bound is asymptotically achieved by a sequence of thinning rectangles and the upper bound by a sequence of thinning triangles. The lower bound in \eqref{polyatorsion} was generalized to higherdimension in \cite{gavitone_2014}, proving it for any open, bounded, and convex sets in $\mathbb{R}^n$ and showing that it is asymptotically achieved by a sequence of flattening cylinders, i.e. a sequence whose height tends to zero. 
Further generalizations can be found in \cite{BBP2021,AMPS}.\\
With regards to \textbf{(ii)}, the following bounds are known 
\begin{equation}\label{polyaeigenvalue}
    \frac{\pi^2}{4n^2}\le\frac{\lambda( \Omega) \abs{\Omega}^2}{P^2(\Omega)}\le \frac{\pi^2}{4}.
\end{equation}
The upper bound was first proved in \cite{polya1960} in the class of convex planar sets,
being sharp for a sequence of thinning rectangles. Successively, it was generalized to higher dimensions and in the case of the first eigenvalue of the anisotropic $p$-Laplace operator by \cite{gavitone_2014}. The lower bound was proved in \cite{makai} and in \cite{brasco_inradius} in two and n dimensions, respectively, and it is sharp on sequences of collapsing pyramids (see \cite{DBN} for a generalization in the anisotropic setting).\\
Concerning \textbf{(iii)}, the following bounds are known 
\begin{equation}
\label{hershtype}
\frac{1}{n(n+2)} \le \frac{T(\Omega)}{R_{\Omega}^2\abs{\Omega}} \leq \frac{1}{3}.
\end{equation}
Makai proved (see \cite{makai}) the upper bound in the two dimensional setting, 
also proving the sharpness for sequences of thinning rectangles. The lower bound was first proved in \cite{PS}, where the equality holds if and only if $\Omega$ is a disc.
Later on, in \cite{dellapietra_gavitone2018} the authors generalized both inequalities in higher dimension and also for more general operators. Among their results, they prove that the upper bound in \eqref{hershtype} is achieved by a suitable sequence of flattening cylinders.\\
Finally, for the functional in \textbf{(iv)} we have the following bounds
\begin{equation}\label{eq:Hersch-Protter}
    \frac{\pi^2}{4}\le\lambda( \Omega)R^2_\Omega\le \lambda_1(B_1).
\end{equation}
The upper bound is an immediate consequence of monotonicity with respect to the inclusion. The lower bound is known as \textit{Hersch-Protter inequality}, since it has been originally proved by Hersch \cite{hersch1960frequence} in the two dimensional case and, later on, generalized in higher dimension by Protter \cite{protter1981lower}.

Further generalizations have been developed for more general spectral functionals, such as the principal frequencies of the p-Laplacian and related shape functionals. Contributions in this direction include \cite{brasco2018principal, brasco_2020_principal_frequencies}.

Besides the lower bound in \eqref{hershtype} and the upper bound in \eqref{eq:Hersch-Protter}, the common thread of all these functionals is that their bounds are achieved by sequences of particular sets, without admitting an optimal set, i.e. if $J(\cdot)$ is any of the functionals in \eqref{eq:functionals} and $\mathcal{K}_n$ is
$$\mathcal{K}_n= \left\{\text{non-empty, bounded, open, and convex subset of }\R^n\right\},$$ then
$$
\nexists \,\tilde \Omega\in\mathcal{K}_n \text{ such that } \, \inf_{\Omega \in \mathcal{K}_n} \mathcal{J}(\Omega) = \mathcal{J}(\tilde \Omega).
$$ 
The aim of this paper is to prove quantitative results for the shape functionals in \eqref{eq:functionals}.\\

The quantitative results for the functionals defined in \eqref{eq:functionals} differ substantially  from the classical ones known in literature, due to the existence of an optimum. Indeed, for such cases, the quantitative analysis becomes more complex, as there is no optimal set to compare the minimising sequence with. 

Quantitative stability problems of this type have been investigated in the literature.
For spectral functionals involving the gap between the first two eigenvalues, both in the Dirichlet and Neumann settings, we refer to the works \cite{ABF,ABF2}, where quantitative inequalities are established. 
Quantitative results for shape functionals related to the Cheeger constant have been obtained in \cite{FT2021,FMP}, where stability estimates are derived.

A different but closely related approach is provided by quantitative versions of classical isoperimetric-type inequalities involving spectral quantities and torsional rigidity. In particular, in \cite{VFNT2016} the authors study a Pólya-type functional combining the first Dirichlet eigenvalue, the torsional rigidity, and the measure of the domain, obtaining quantitative lower bounds and stability results.

It will soon be clear that the remainder terms that we will add to the qualitative inequalities will not fully characterize the shape of the minimizing sequence, as for instance for the isoperimetric inequality, but they will allow to give some information and properties of such sequences.

 Let us now define the  
following two remainder terms 
\begin{equation}\label{eq:remterms}
    \alpha(\Omega):=\frac{w_\Omega}{\diam(\Omega)}, \qquad \qquad \text{and} \qquad \qquad  \beta(\Omega):=\frac{P(\Omega)R_\Omega}{\abs{\Omega}}-1,
\end{equation}
where we denote by $w_\Omega$ and $\diam(\Omega)$, the minimal width and the diameter of $\Omega$, respectively (see Section \ref{section_notion} for the exact definitions). 
We stress that the first remainder term allows to define the class of the so-called \textit{thinning domains} (see Section \ref{sec:2} for the precise definition), that are sequences of sets for which $\alpha(\Omega)\to 0$. Moreover the remainder term $\beta(\Omega)$ is always between $0$ and $n-1$, where the lower bound is sharp and which is achieved by a sequence of flattening cylinders, meanwhile the upper bound is sharp, for instance, on balls (see Proposition \ref{prop:PRM-1}). \\
First, we emphasize a result connecting the two remainder terms $\alpha(\Omega)$ and $\beta(\Omega)$.
\begin{prop}\label{prop:comparisonasymmetries}
     Let $\Omega \in \mathcal{K}_n$. There exists a positive dimensional constant $C_o(n)$ such that
     \begin{equation}\label{eq:PRM-WD}
         \beta(\Omega)\ge C_0(n)\alpha(\Omega).
     \end{equation}
The constant $C_0(n)$ can be explicitly computed. Moreover, the exponent of $\alpha(\Omega)$ is sharp and a reverse inequality cannot be true, since there are sequences of thinning domains for which the functional $\beta(\Omega)$ is not converging to zero (for instance a sequence of collapsing pyramids).
\end{prop}

This proposition provides information about the nature of the minima of the functional $\beta(\Omega)$: every minimizing sequence for $\beta$ must consist of thinning domains.
The latter shows that the two remainder terms are not equivalent, in the sense that $\beta(\Omega)$ cannot be bounded from above in terms of $\alpha(\Omega)$. This is the reason why we distinguish the quantitative results for the functionals in \eqref{eq:functionals} with respect to $\alpha(\Omega)$ and $\beta(\Omega)$.

 Finally, we highlight that, when referring to the sharp exponent, we mean that there exists a sequence for which the functional on the left-hand side divided by the right-hand side converges to a positive constant.\\
 The purpose of this paper is, in a sense, to provide a characterization of the minimizing sequences for the lower bounds in \eqref{polyatorsion} and \eqref{eq:Hersch-Protter} and for the upper bounds in \eqref{polyaeigenvalue} and \eqref{hershtype}, by means of quantitative and continuity results expressed in terms of $\alpha(\Omega)$ or $\beta(\Omega)$. More precisely, we aim to establish inequalities that ensure two key properties: first, if the functional is close to its optimum, then the corresponding shape is close to a slab; and second, if the set is close to a slab, then the functional is close to its infimum.

An initial step in this direction appears in \cite{AMPS}, where two quantitative results are established for the functional in \textbf{(i)}.
The first one involves $\alpha(\Omega)$: in any dimension and it holds
\begin{equation}\label{eq:AMPS1}
    \frac{T(\Omega) P^2(\Omega)}{\abs{\Omega}^3}-\frac{1}{3}\ge K_1(n)\alpha(\Omega)^{n-1},
\end{equation}
where $K_1(n)$ is a positive dimensional constant and $\Omega\in \mathcal{K}_n$. Inequality \eqref{eq:AMPS1} states that when the functional is close to its optimal value, the domain $\Omega$ must be a thinning domain.

 To have more information on the shape of the minimizing sequence, the authors proved in dimension $2$ that there exists a  positive constant $K_2$ such that
\begin{equation}\label{eq:AMPS3}
\frac{T(\Omega) P^2(\Omega)}{\abs{\Omega}^3}-\frac{1}{3}\ \ge K_2 \beta(\Omega)^3.    
\end{equation}

The first result is the $n$-dimensional generalization and improvement to the inequality \eqref{eq:AMPS3} regarding the P\'olya torsion functional.

\begin{teorema} \label{thm:quantitative_PRM}
Let $\Omega\in \mathcal{K}_n$. Then, 
 \begin{equation}\label{eq:TPM-PRM}
         \frac{n+1}{3}\beta(\Omega)\ge 
         \frac{T(\Omega) P^2(\Omega)}{\abs{\Omega}^3}-\frac{1}{3} \ge \frac{1}{2^3\cdot 3^4 n^3} \beta(\Omega)^3.
    \end{equation}
\end{teorema}
\vspace{0.5cm}
\noindent The second one is about the P\'olya eigenvalue functional.

\begin{teorema}\label{thm:quantitative_lambdapr}
Let $\Omega\in \mathcal{K}_n$. Then
\begin{equation}\label{quantitative_HMP-PRM}
\frac{\pi^2}{2}\beta(\Omega)\ge  \frac{\pi^2}{4} -\frac{\lambda( \Omega) \abs{\Omega}^2}{P^2(\Omega)}  \ge \frac{\pi^2}{2^5\cdot 3^4}\cdot\frac{1}{n^3(2n-1)}\beta(\Omega)^4.
\end{equation}

\end{teorema}
\vspace{0.5cm}
\noindent The third main results regards the Makai functional.
\begin{teorema} \label{thm:quantitative_torsionmakai_PRM}
Let $\Omega\in \mathcal{K}_n$. Then
\begin{equation}\label{eq:quantitative_PRM_torsion_makai}
\frac{2}{3}\beta(\Omega)\ge\frac{1}{3}-\frac{T(\Omega)}{R_{\Omega}^2\abs{\Omega}}>\frac{n+1}{3n(2n-1)}\beta(\Omega).
\end{equation}
\end{teorema}
\begin{oss}\label{rem:betasharp} Some comments on the results are in order:
\begin{itemize}
    \item The main idea behind the proofs of Theorems~\ref{thm:quantitative_PRM} and~\ref{thm:quantitative_lambdapr} is to use a suitable test function, following the approach originally introduced by P\'olya, and then to carefully handle the terms that were discarded in the classical argument. It turns out that the measure and the perimeter of a specific superlevel set of the distance function, namely $\Omega_{P(\Omega)/|\Omega|}$, play a crucial role in this analysis. This is the reason why Lemma~\ref{lemma:M(PM)eP(MP)} is needed.
 In the proof of Theorem~\ref{thm:quantitative_torsionmakai_PRM}, we rely on estimates of the torsion in terms of the 
$L^2$-norm of the distance function to the boundary (see Proposition \ref{prop:estimatestorsion}), together with concavity properties of the perimeter of the superlevel sets of the distance function.
    \item Theorems \ref{thm:quantitative_PRM}, \ref{thm:quantitative_lambdapr} and \ref{thm:quantitative_torsionmakai_PRM}  fully characterize the optimizing sequences of these three functionals in terms of the functional $\beta(\Omega)$, which is purely geometric. Specifically, these results allow us to assert that a sequence is minimizing for the \emph{P\'olya torsion functional}, \emph{P\'olya eigenvalue functional}, or \emph{Makai functional} if and only if it minimizes $\beta(\Omega)$.

    \item The optimality of the exponents of $\beta(\Omega)$ in Theorem~\ref{thm:quantitative_torsionmakai_PRM}  follows directly from the Theorem itself, since the two exponents are equal. 
Indeed, for any set for which $\beta$ tends to $0$, the Makai functional also goes to $\frac{1}{3}$ with the same order.
\end{itemize}
\end{oss}

\noindent As an immediate consequence of Proposition \ref{prop:comparisonasymmetries} and Theorems \ref{thm:quantitative_PRM}, \ref{thm:quantitative_lambdapr} and \ref{thm:quantitative_torsionmakai_PRM}, we can also bound from below the three quantities in \eqref{eq:TPM-PRM}, \eqref{quantitative_HMP-PRM} and \eqref{eq:quantitative_PRM_torsion_makai} in terms of $\alpha(\Omega)$, with the same exponents. Therefore, regarding the P\'olya torsion functional, we have an improvement of the inequality \eqref{eq:AMPS1} for $n\ge 5$. The authors of \cite{AMPS} conjectured that the optimal exponent in \eqref{eq:AMPS1} is $1$. We state and prove this conjecture, and prove an analogous result for the P\'olya eigenvalue functional. This is the content of the following proposition.
\begin{prop} \label{prop:quantitative_wd}
Let $\Omega\in \mathcal{K}_n$. Then, 
\begin{equation}\label{quantitative_width}
 \frac{T(\Omega) P^2(\Omega)}{\abs{\Omega}^3}-\frac{1}{3}  \ge C_1(n) \alpha(\Omega),
\end{equation}
and 
\begin{equation}\label{quantitative_width2}
 \frac{\pi^2}{4} - \frac{\lambda( \Omega) \abs{\Omega}^2}{P^2(\Omega)}
 \ge C_2(n) \alpha(\Omega),
\end{equation}
where $C_1(n)$ and $C_2(n)$ are positive constants depending only on the dimension $n$, and they can be explicitly computed.  
In particular, the exponents of $\alpha(\Omega)$ are sharp. Moreover, a reverse inequality cannot hold in either case, since there exist sequences of thinning domains for which the left-hand side does not tend to zero (for instance, sequences of collapsing pyramids).

\end{prop}  

For the sake of completeness, we recall some known-in-literature result and remarks on the Hersch functional \textbf{(iv)} defined in \eqref{eq:functionals}. It is known that for any  $\Omega\in \mathcal{K}_n$,
\begin{equation}\label{eq:protter_WD}
    K_3(\diam(\Omega)^2\Lambda)\alpha(\Omega)^\frac{2}{3}\ge \lambda( \Omega)R^2_\Omega-\frac{\pi^2}{4}\ge K_4\alpha(\Omega)^2,
\end{equation}
for some positive real constants $K_3(\diam (\Omega)^2\Lambda)$ and $K_4$, where $\Lambda$ denotes the first eigenvalue of the $(m-1)$-dimensional Dirichlet Laplacian of the projection of $\Omega$ onto the hyperplane orthogonal to the direction of the width (see \cite{VFNT2016} for further details).
 The lower bound was proved in \cite{protter1981lower}, while the upper bound is actually hidden in the proof of \cite[Theorem $1.5$]{VFNT2016}. In dimension $2$, the chain of inequalities \eqref{eq:protter_WD} completely characterizes the minimizers of the Hersch functional, which are precisely the thinning domains (and vice versa), since in this case it is possible to prove the boundedness of $\Lambda$.
 An upper bound has also been proved in \cite{FT2021}, albeit with a different approach.


Inequalities \eqref{eq:protter_WD} show that the remainder term $\beta(\Omega)$ cannot be added as a lower bound; indeed, otherwise one would be able to control $\beta(\Omega)$ from above in terms of $\alpha$, which is not possible by Proposition \ref{prop:comparisonasymmetries}.
Nevertheless, Proposition \ref{prop:comparisonasymmetries} permits its addition above, even though the constant depends on the first Dirichlet eigenvalue of a $(n-1)$-dimensional convex set. To get rid of this problem, we pay a price in terms of exponent, but we can prove that
\begin{corollario}
\label{cor:quantitative_lr}
  Let $\Omega\in \mathcal{K}_n$. Then, 
\begin{equation}\label{quantitative_LR}
 \frac{\pi^2(n+1)}{4}\beta(\Omega)\ge\lambda( \Omega)R^2_\Omega-\frac{\pi^2}{4}.
\end{equation}
\end{corollario}
The proof is just a manipulation of the functional and the use of the lower bound for the P\'olya eigenvalue functional.\\

\noindent For the sake of readability, we summarize the new and known results in the following table.

\begin{table}[!ht]
    \caption{Summary of the results.}\label{tab:tab0}

\begin{tabular}{|c||c|c||c|c|}

    \hline
    \multicolumn{1}{|c||}{} & \multicolumn{2}{|c||}{Lower remainder term} & \multicolumn{2}{|c|}{Upper remainder term} \\
    \hline\hline
    & & & &\\[-1.8ex]
    & $\displaystyle\bigg(\frac{w_\Omega}{d_\Omega}\bigg)^\gamma$ &  $\displaystyle\bigg(\frac{P(\Omega)R_\Omega}{\abs{\Omega}}-1\bigg)^\delta$  &  $\displaystyle\bigg(\frac{w_\Omega}{d_\Omega}\bigg)^\gamma$ & {$\displaystyle\bigg(\frac{P(\Omega)R_\Omega}{\abs{\Omega}}-1\bigg)^\delta$}
    \tabularnewline[3ex]
    \hline\hline
    & & & & \\[-1.8ex]
    $\displaystyle \frac{T(\Omega)P^2(\Omega)}{\abs{\Omega}^3}-\frac{1}{3}$  &\makecell{\textcolor{OliveGreen}{\checkmark}\\ $\gamma=1$  \\ sharp} & \makecell{\textcolor{OliveGreen}{\checkmark}\\ $\delta=3$}  & \makecell{\textcolor{red}{\ding{55}}\\ (collapsing pyramids)}& \makecell{\textcolor{OliveGreen}{\checkmark} \\ $\delta=1$ \\ sharp} \tabularnewline[3ex]
    \hline 
    & & & & \\[-1.8ex]
    $\displaystyle\frac{\pi^2}{4}-\frac{\lambda(\Omega)\abs{\Omega}^2}{P^2(\Omega)}$ &  \makecell{\textcolor{OliveGreen}{\checkmark}\\ $\gamma=1$ \\ sharp} & \makecell{\textcolor{OliveGreen}{$\checkmark$} \\ $\delta=4$}&  \makecell{\textcolor{red}{\ding{55}}\\ (collapsing pyramids)} &  \makecell{\textcolor{OliveGreen}{\checkmark} \\ $\delta=1$ \\ sharp} 
    \tabularnewline[3ex]
    \hline
    & & & & \\[-1.8ex]
    $\displaystyle \frac{1}{3}- \frac{T(\Omega)}{R_\Omega^2 \abs{\Omega}}$ &\makecell{\textcolor{OliveGreen}{\checkmark}\\ $\gamma=1$\\
    sharp} & \makecell{{\textcolor{OliveGreen}{\checkmark}}\\ $\delta=1$\\
    sharp} & \makecell{\textcolor{red}{\ding{55}}\\ (collapsing pyramids)} &  \makecell{\textcolor{OliveGreen}{\checkmark} \\ $\delta=1$ \\ sharp 
    } 
    \tabularnewline[3ex]
    \hline
    & & & & \\[-1.8ex]
    $\displaystyle \lambda(\Omega)R_\Omega^2-\frac{\pi^2}{4}$ &\makecell{\textcolor{black!65}{\faBook}\\ $\gamma=2$} & \makecell{\textcolor{red}{\ding{55}}\\ (collapsing pyramids)} & \makecell{\textcolor{black!65}{\faBook} \\ $\gamma=2/3$}&  \makecell{\textcolor{OliveGreen}{\checkmark} \\ $\delta=1$
    } 
    \tabularnewline[3ex]
        \hline
\end{tabular}
\begin{quotation}
    \textbf{Symbol Legend:}

    \begin{itemize}
  \item[\textcolor{OliveGreen}{\checkmark}]: Proved in this paper in any dimension;
  \item[\textcolor{red}{\ding{55}}]: Not possible (we indicate the counterexample);
    \item[\textcolor{black!65}{\faBook}]: Known in literature.

    \end{itemize}
  \end{quotation}
\end{table}
\vspace{0.7cm}

\noindent \textbf{Plan of the paper:} In Section $2$ we recall some basic notions and definitions, and we recall some classical results, focusing in particular on the class of convex sets. In Section $3$ we prove Theorems \ref{thm:quantitative_PRM}, \ref{thm:quantitative_lambdapr} and \ref{thm:quantitative_torsionmakai_PRM} about the estimates involving $\beta(\Omega)$, while in  Section $4$ we give the proof of Proposition \ref{prop:quantitative_wd}, regarding the one involving $\alpha(\Omega)$. Eventually, Section $5$ is dedicated to the proof of the sharpness of the inequalities proved.

\section{Preliminary results}\label{sec:2}
\label{section_notion}
\subsection{Notations and basic facts} 
 Throughout this article, $|\cdot|$ will denote the Euclidean norm in $\mathbb{R}^n$,
 while $(\,\cdot\,)$ is the standard Euclidean scalar product for  $n\geq2$. For $k\in [0,n)$, the $k-$dimensional Hausdorff measure is denoted by $\mathcal{H}^k(\cdot)$. We will denote by $B_r(x)$ the ball centered at the point $x\in \mathbb R^n$ with radius $r>0$; moreover, when the ball is centered at the origin we will write $B_r$, omitting the dependence on $x$, and when the radius of the ball is $1$, we will denote by $\mathbb{S}^{n-1}$ its boundary.
 
Since $\Omega \in \mathcal{K}_n$, its perimeter can be defined by
\[
P(\Omega) = \mathcal{H}^{n-1}(\partial \Omega).
\]
If $P(\Omega) < \infty$, we say that $\Omega$ is a set of finite perimeter. Some references for results relative to the sets of finite perimeter and for the coarea formula are, for instance, \cite{ambrosio2000functions,maggi2012sets}.

 We give now the definition of the support function of a convex set and minimal width (or thickness) of a convex set.
\begin{definizione}\label{support}
  Let $\Omega\in \mathcal{K}_n$. The support function of $\Omega$ is defined as
  \begin{equation*}
    h_\Omega(y)=\sup_{x\in \Omega}\left(x\cdot y\right), \qquad y\in \mathbb{R}^n .
  \end{equation*}
\end{definizione}

\begin{definizione}
Let $\Omega\in \mathcal{K}_n$, the width of $\Omega$ in the direction $y \in \mathbb{S}^{n-1}$ is defined as 
  \begin{equation*}
    \omega_{\Omega}(y)=h_{\Omega}(y)+h_{\Omega}(-y)
    \end{equation*}
 and the minimal width of $\Omega$ as
\begin{equation*}
    w_\Omega=\min\{  \omega_{\Omega}(y)\,|\; y\in\mathbb{S}^{n-1}\}.
\end{equation*}
\end{definizione}
We will denote by $R_\Omega$ is the inradius of $\Omega$, i.e.
 \begin{equation}
     \label{inradiuss}
     R_\Omega=\sup\{r\in \mathbb R: B_r(x)\subset \Omega, x\in \Omega\},
 \end{equation}
and by $\diam(\Omega)$ the diameter of $\Omega$, that is
\begin{equation*}
    \diam(\Omega) = \sup_{x,y\in\Omega}|x-y|.
\end{equation*}
\begin{definizione}
Let $\Omega\in \mathcal{K}_n$. We say that $\Omega_l$ is a sequence of thinning domains if
    \begin{equation*}
        \dfrac{w_{\Omega_l}}{\diam(\Omega_l)}\xrightarrow{l \to 0}0.
    \end{equation*}
See \cite{AMPS} for more details and some figures.
\end{definizione}

We recall in the following the relation between the inradius and the minimal width (see as a reference \cite{ scott_family, santalo_sobre,fenchel_bonnesen}).
\begin{prop}
\label{prop:r>w}
Let $\Omega\in \mathcal{K}_n$. Then,
the following estimates (that can be found in \cite[equation (9)]{fenchel_bonnesen}):
\begin{equation}
\label{eq:lowboundinradius}
    \displaystyle \frac{w_\Omega}{2}\ge R_{\Omega}\ge\begin{cases}
    w_{\Omega} \displaystyle{\frac{\sqrt{n+2}}{2n+2}} & n \,\, \text{even}\\ \\
    w_{\Omega} \displaystyle{\frac{1}{2\sqrt{n}}} & n \,\, \text{odd},
    \end{cases}
\end{equation}
\end{prop}

Moreover, we have the following estimate involving the perimeter and the diameter, which follows from \cite[equation (7)]{fenchel_bonnesen},
\begin{equation}\label{eq:perdiam}
\displaystyle{P(\Omega)\le n\omega_n \left(\frac{n}{2n+2}\right)^{\frac{n-1}{2}} \diam(\Omega)^{n-1}}.
 \end{equation}

 \subsection{Properties of the distance function and inner parallel sets} \label{innere}
Let $\Omega\in \mathcal{K}_n$. We define the distance function from the boundary, and we will denote it by $ d(\cdot, \partial \Omega):\Omega \to [0,+\infty[$, as follows 
 \begin{equation*}
     d(x,\partial\Omega):=\inf_{y\in\partial\Omega}\abs{x-y},
 \end{equation*}
 and we call inradius, $R_\Omega$, its maximum.
 
We remark that the distance function is concave, as a consequence of the convexity of $\Omega$.
The superlevel sets of the distance function
\begin{equation*}
    \Omega_t=\set{x\in \Omega \, : \, d(x,\partial\Omega)>t}, \qquad t\in[0, R_\Omega]
\end{equation*}
 are called \emph{inner parallel sets}, and we use the following notations:
 \begin{equation*}
 \mu(t)= \abs{\Omega_t}, \qquad P(t)=P(\Omega_t)\qquad t\in[0, R_\Omega].
 \end{equation*}
 By coarea formula, recalling that $\abs{\nabla d}=1$ almost everywhere, we have 
 $$\mu(t)=\int_{\Set{d>t}} \,  dx = \int_{\Set{d>t}} \frac{\abs{\nabla d}}{\abs{\nabla d}} \,  dx= \int_{t}^{R_\Omega} \frac{1}{\abs{\nabla d}} \int_{\Set{d=s}} d\mathcal{H}^{n-1} \, ds= \int_{t}^{R_\Omega} P(s)\; ds;$$
hence, the function $\mu(t)$ is absolutely continuous, decreasing and its derivative is  
\begin{equation}\label{eq:dermu}
    \mu'(t)=-P(t)\qquad a.e.
\end{equation}
By the Brunn-Minkowski inequality (\cite[Theorem 7.4.5]{schneider}) and the concavity of the distance function, the map
\begin{equation*}
    t \mapsto P(t)^{\frac{1}{n-1}}
\end{equation*}
is concave in $[0,R_{\Omega}]$, hence absolutely continuous in $(0,R_{\Omega})$. Moreover, there exists its right derivative at $0$ and it is negative, since $P(t)^{\frac{1}{n-1}}$ is strictly monotone decreasing, hence almost everywhere differentiable. As a consequence of the monotonicity of $P(t)^{\frac{1}{n-1}}$, also $P(t)$ is strictly monotone decreasing, and then, by \eqref{eq:dermu}, $\mu(t)$ is convex. Moreover the concavity allows us to say that $(P^\frac{1}{n-1}(t))''\le 0$. Furthermore, integrating \eqref{eq:dermu} from $0$ to $\abs{\Omega}$ and considering the fact that in convex sets $P(t)\le P(\Omega)$, we get
\begin{equation}
\label{steiner-measure}
    \mu(t) \ge \abs{\Omega}-P(\Omega)t \qquad a.e.
\end{equation}
The well-known Steiner formula for outer parallel sets reads
\[
|\Omega+\rho B_1| =\sum_{i=0}^{n} \binom{n}{i} W_i(\Omega) \rho^i.
\]
The coefficient $W_i(\Omega)$, $i=0,\ldots,n$ is known as the $i$-th
quermassintegral of $\Omega$. It is well know that $W_0(\Omega)=|\Omega|$,
$nW_1(\Omega)=P(\Omega)$, $W_n(\Omega)=\omega_n$, where $\omega_n$ is the measure of the unit ball in dimension $n$. If $\Omega$ is of class $C^2$, with
nonvanishing Gaussian curvature, the quermassintegrals can be connected
to the principal curvatures of the boundary of $\Omega$.

Crucial to mention are the Aleksandrov-Fenchel inequalities 
\begin{equation*}
\left( \frac{W_j(\Omega)}{\omega_n} \right)^{\frac{1}{n-j}} \ge \left(
  \frac{W_i(\Omega)}{\omega_n} \right)^{\frac{1}{n-i}}, \quad 0\le i < j
\le n-1,
\end{equation*}
with equality sign if and only if  $\Omega$ is
a ball. As a consequence of the Alexandrov-Fenchel inequality, we have the following lemma, proved in \cite{BNT2010}.
\begin{lemma}
\label{lema22}
Let $\Omega\in \mathcal{K}_n$. Then for a.e. $s\in (0,R_\Omega)$ we have
    \begin{equation*}
        -\frac{d}{ds}P(s) \ge (n-1) W_2(s),
    \end{equation*}
and the equality holds if and only if $\Omega$ is a ball.
Hence, for all $s\in (0,R_\Omega)$ we have  \begin{equation}\label{eq:PerW2}
       P(s) \le P(\Omega)-(n-1)\int_0^s W_2(s)\,ds.
    \end{equation}
\end{lemma}
 In particular, for $\Omega$ non-empty bounded, open and convex set of $\mathbb{R}^2$, \eqref{eq:PerW2} reads
  \begin{equation}\label{steiner1}
     P(t)\leq P(\Omega)-2\pi t \qquad \forall t \in [0,R_\Omega].
 \end{equation}
 equality holding in \eqref{steiner1} for the stadii (see \cite{fragala_gazzola_lamboley}).

The following lemma is a key ingredient for the quantitative estimates. The first estimate will be crucial for proving the quantitative bounds involving $\alpha(\Omega)$ for the P\'olya torsion and eigenvalue functionals, while the second one will be used to establish the estimates involving $\beta(\Omega)$.
\begin{lemma}
Let $\Omega\in \mathcal{K}_n$. Then, for all $t\in (0,R_\Omega)$, we have

\begin{equation}
    \label{perimeter_estimate}
    P(t)\leq P(\Omega) - c_n \frac{\abs{\Omega} - \mu(t)}{P^{\frac{1}{n-1}}(\Omega)}, \text{ for } c_n=(n-1)\omega_n^{\frac{1}{n-1}},
\end{equation}
and 

\begin{equation}
    \label{measure_estimate1}
    \mu(t)\leq (R_\Omega-t)P(t)+ \frac{(R_\Omega-t)^2}{2(n-1)}P'(t).
\end{equation}
\end{lemma}
\begin{proof}We begin with the proof of \eqref{perimeter_estimate}.  
In the planar case, this inequality is a consequence of \eqref{steiner1} and \eqref{steiner-measure}, with $c_2 = 2\pi$.

If $n\geq 3$, by Lemma \ref{lema22} and the Alexandrov-Frenchel inequalities, we have
\begin{align*}
          P(t) &\le P(\Omega)-(n-1)\int_0^tW_2(s)\,ds\\
         &\leq P(\Omega)- c_n \int_0^t P^{\frac{n-2}{n-1}}(s)\, ds\\
        &\leq P(\Omega)- c_n P^{-\frac{1}{n-1}}(\Omega)\int_0^t -\mu'(s)\, ds\\
         &=P(\Omega) - c_n \frac{\abs{\Omega} - \mu(t)}{P^{\frac{1}{n-1}}(\Omega)},
    \end{align*}
    where $c_n=(n-1)\omega_n^{\frac{1}{n-1}}$.

We now prove \eqref{measure_estimate1}.  
Since the measure of the inner parallel sets is absolutely continuous and the perimeter is monotone, we have

    \begin{equation*}
        \mu(t) = \int_{t}^{R_\Omega}P(s)\,ds = \int_{t}^{R_\Omega}P(s)^{\frac{1}{n-1}}P(s)^{\frac{n-2}{n-1}}\,ds\le P(t)^{\frac{n-2}{n-1}}\int_{t}^{R_\Omega}P(s)^{\frac{1}{n-1}}\,ds.
    \end{equation*}
    Now, integrating by parts twice, we get
    \begin{equation*}
    \begin{split}
    \int_{t}^{R_\Omega}P(s)^{\frac{1}{n-1}}\,ds &= (R_\Omega-t)P(t)^{\frac{1}{n-1}}+\int_t^{R_\Omega}(R_\Omega-s)(P(s)^{\frac{1}{n-1}})'\,ds \\
    &=(R_\Omega-t)P(t)^{\frac{1}{n-1}}+\frac{(R_\Omega-t)^2}{2}(P(t)^{\frac{1}{n-1}})'+\frac{1}{2}\int_t^{R_\Omega}(R_\Omega-s)^2(P(s)^{\frac{1}{n-1}})''\,ds.
    \end{split}
\end{equation*}
Since the map $t\to P(t)^{\frac{1}{n-1}}$ is concave, $(P(s)^{\frac{1}{n-1}})''\le 0$ for a.e. $s$, and we get
\begin{equation*}
  \int_{t}^{R_\Omega}P(s)^{\frac{1}{n-1}}\,ds\le (R_\Omega-t)P(t)^{\frac{1}{n-1}}+\frac{(R_\Omega-t)^2}{2(n-1)}P(t)^{-\frac{n-2}{n-1}}P'(t)  . 
\end{equation*}
    Therefore
    \begin{equation*}
     \mu(t)\leq (R_\Omega-t)P(t)+ \frac{(R_\Omega-t)^2}{2(n-1)}P'(t).   
    \end{equation*}
\end{proof}

\subsection{Upper and lower bounds for the Torsion}\label{subsec:webtors}

Two key ingredients for the proofs of the lower bounds for the P\'olya and Makai functionals are specific estimates below and above of the torsion $T(\Omega)$, that we here explain. 

Regarding the lower bound,
the idea has been known for several decades and can be found in \cite{polya1960}. Here we rewrite the computations made by P\'olya to give a better comprehension of the subject and to make it more readable for non-expert readers. 

\begin{prop}\label{prop:estimatestorsion}
    Let $\Omega\in \mathcal{K}_n$ and let us consider the torsion $T(\Omega)$ of the set $\Omega$. Then
    
\begin{equation*}
\label{eq:lowboundpolya}
\int_0^{R_\Omega}\frac{\mu^2(t)}{P(t)}\,dt\leq T(\Omega) \leq \int_{\Omega}d(x,\partial\Omega)^2\,dx.
\end{equation*}
\end{prop}

\begin{proof}
    Let us prove only the lower bound since we will follow the same arguments proposed by P\'olya, i.e. the use of a suitable test function of the distance from the boundary.
    For the proof of the upper bound we refer to Makai (see \cite{makai}) for the planar case and to \cite[Theorem 1.1]{prinari2023sharp}, for the general case.
    
    Let us consider as a test function for the torsion $f(x)= g(d(x,\partial\Omega))$,where $g$ is suitably chosen. At this point, coarea formula and an integration by parts allow to write
$$
\int_\Omega f(x)\, dx = \int_0^{R_\Omega} g(t) P(t) \, dt=\int_0^{R_\Omega} g'(t) \mu(t) \, dt,$$
and
\begin{equation*}
    \int_\Omega \abs{ \nabla f}^2\,dx = \int_0^{R_\Omega} g'^2(t) P(t) \, dt.
\end{equation*}
In this way, we get

\begin{equation*}
    \label{torsionpolya}
    T(\Omega) \geq \frac{\displaystyle{\left(\int_\Omega f(x)\, dx\right)^2}}{\displaystyle{\int_\Omega \abs{ \nabla f}^2\, dx}} =\frac{\displaystyle{\left(\int_0^{R_\Omega} g'(t) \mu(t) \, dt\right)^2}}{\displaystyle{\int_0^{R_\Omega} g'^2(t) P(t) \, dt}},
\end{equation*}
and choosing $g'(t) = \mu(t)/P(t)$, we finally have 
\begin{equation*}
\int_0^{R_\Omega}\frac{\mu^2(t)}{P(t)}\,dt \leq T(\Omega) .
\end{equation*}
\end{proof}

The integral on the left-hand side of \eqref{eq:lowboundpolya} is the exact representation of the so-called web torsion (see \cite{PS}). See also \cite{CFG2002} for the upper estimate in \eqref{eq:lowboundpolya} proved in dimension $2$, and successively generalized in higher dimensions in \cite{Cra2004}.

\subsection{Comparison results between the remainder terms}
As we already mentioned, our aim is to prove quantitative estimates for the functionals in \eqref{eq:functionals}, using as remainder terms the two remainder terms defined in \eqref{eq:remterms}, that we write here again
\begin{equation*}
\alpha(\Omega):=\frac{w_\Omega}{\diam(\Omega)}, \qquad \qquad \text{and} \qquad \qquad  \beta(\Omega):=\frac{P(\Omega)R_\Omega}{\abs{\Omega}}-1.
\end{equation*}
We recall the following estimate, which is proved in \cite{fenchel_bonnesen} in the planar case and is generalized in \cite{DBN,Kova} to all dimensions. 

\begin{prop}\label{prop:PRM-1}
Let $\Omega\in \mathcal{K}_n$. Then, 
 \begin{equation}\label{convex_estimates}
1  < \dfrac{P(\Omega) R_\Omega}{|\Omega|}  \le n. 
 \end{equation}
The lower bound is sharp on a sequence of flattening cylinders, while the upper bound is achieved if $\Omega_{(1-t)R_{\Omega}}= t \Omega $ for $t \in (0,1)$. 
\end{prop}

\noindent Our first result is a quantitative version of the lower bound in \eqref{convex_estimates} in terms of $\alpha(\Omega)$.

\begin{proof}[Proof of Proposition \ref{prop:comparisonasymmetries}]
If we integrate the estimate \eqref{perimeter_estimate} between $0$ and the inradius $R_\Omega$, and use the fact that $-\mu'(t)=P(t)$, we obtain
\begin{equation*}
\begin{split}
    \abs{\Omega}= \int_0^{R_\Omega}P(t)\,dt&\le P(\Omega)R_\Omega-c_n \int_0^{R_\Omega}\frac{\abs{\Omega} - \mu(t)}{P^{\frac{1}{n-1}}(\Omega)}\,dt\\
    &=P(\Omega)R_\Omega-c_n\frac{\abs{\Omega}}{P^{\frac{1}{n-1}}(\Omega)}\int_0^{R_\Omega}\bigg(1-\frac{\mu(t)}{\abs{\Omega}}\bigg)\,dt\\
    &\le P(\Omega)R_\Omega-c_n\frac{\abs{\Omega}^2}{P^{\frac{1}{n-1}}(\Omega)P(\Omega)}\int_0^{R_\Omega}\bigg(1-\frac{\mu(t)}{\abs{\Omega}}\bigg)\frac{-\mu'(t)}{\abs{\Omega}}\,dt\\
    &= P(\Omega)R_\Omega-\frac{c_n}{2}\frac{\abs{\Omega}^2}{P^{\frac{1}{n-1}}(\Omega)P(\Omega)}.
    \end{split}
\end{equation*}
Now, dividing by $\abs{\Omega}$ and using estimate \eqref{convex_estimates} we have
\begin{equation*}
    \begin{split}
        \frac{P(\Omega)R_\Omega}{\abs{\Omega}}-1 \ge \frac{c_n}{2}\frac{\abs{\Omega}}{P^{\frac{1}{n-1}}(\Omega)P(\Omega)}\ge \frac{c_n}{2n}\frac{R_\Omega}{P^{\frac{1}{n-1}}(\Omega)}.
    \end{split}
\end{equation*}
Eventually, considering \eqref{eq:lowboundinradius} and \eqref{eq:perdiam}, we obtain
\begin{equation*}
    \beta(\Omega)=\frac{P(\Omega)R_\Omega}{\abs{\Omega}}-1 \ge C_0(n)\frac{w_\Omega}{\diam(\Omega)}= C_0(n)\alpha(\Omega),
\end{equation*}
where
\begin{equation*}
    \displaystyle C_0(n)=\begin{cases}
     \displaystyle{\frac{n-1}{2^{\frac52}n^{\frac{n}{n-1}}}\bigg[\frac{n(n+2)}{(n+1)^3}\bigg]^\frac{1}{2}} & n \,\, \text{even}\\ \\
    \displaystyle{\frac{n-1}{2^{\frac52}n^{\frac{n}{n-1}}}\bigg[\frac{1}{n+1}\bigg]^\frac{1}{2}} & n \,\, \text{odd}.
    \end{cases}
\end{equation*}
For the optimality of the exponent  see Proposition \ref{prop:sharpnessexponent}, while for the counterexample see Remark \ref{rem:counterexamples}.
\end{proof}

\section[Proofs of the results in terms of beta]{Proofs of the results in terms of $\beta$}
For the proof of Theorems \ref{thm:quantitative_PRM} and \ref{thm:quantitative_lambdapr}, we need first to prove a technical Lemma. We consider a particular inner parallel set at the level $\abs{\Omega}/P(\Omega)$. This choice is justified by Proposition \ref{prop:PRM-1}, being this quantity  always less than $R_\Omega$ and tending to $R_\Omega$ when $\beta(\Omega)\to 0$. In order to prove the main results, we need to relate the measure and the perimeter of this inner parallel set  with the remainder term $\beta(\Omega)$. 
\begin{lemma}\label{lemma:M(PM)eP(MP)}
Let $\Omega\in \mathcal{K}_n$ and let $\beta(\Omega)$ be defined in \eqref{eq:remterms}. Then,
    \begin{equation}
        \label{eq:lemmaM(mp)}\mu\left(\frac{\abs{\Omega}}{P(\Omega)}\right)\ge q_1(n,\Omega) \abs{\Omega},
    \end{equation}
\begin{equation}\label{eq:LemmaP(MP)}
        P\left(\frac{\abs{\Omega}}{P(\Omega)}\right)\le q_{2}(n,\Omega)P(\Omega),
    \end{equation}
    where
    \begin{equation}
        q_1(n,\Omega)= \frac{\beta(\Omega)}{6n}
        \quad \text{ and }\quad q_{2}(n,\Omega)= \left(1+\frac{\beta(\Omega)}{n}\right)^{-1}
    \end{equation}
  
\end{lemma}

\begin{proof}
    We start from inequality \eqref{eq:lemmaM(mp)}. Let us assume by contradiction that 
    \begin{equation}\label{absurd}
        \mu\left(\frac{\abs{\Omega}}{P(\Omega)}\right)<\frac{\abs{\Omega}}{6n}\left(\frac{P(\Omega)R_\Omega}{|\Omega|}-1\right)=q_1(n,\Omega)\abs{\Omega}.
    \end{equation}
Let us define
    \begin{equation*}
        \overline{t}=\sup\left\{s\in \left(0,\frac{\abs{\Omega}}{P(\Omega)}\right) : P(s)>\frac{P(\Omega)}{2}\right\}.
    \end{equation*}
By the absolute continuity of $\mu(\cdot)$ and by \eqref{eq:dermu}, we know that 

\begin{equation*}
\mu\left(\frac{\abs{\Omega}}{P(\Omega)}\right)=\mu(\overline{t})+ \int_{\overline{t}}^{\frac{\abs{\Omega}}{P(\Omega)}}\mu'(s) \, ds
    =\mu(\overline{t})- \int_{\overline{t}}^{\frac{\abs{\Omega}}{P(\Omega)}}P(s) \, ds.
\end{equation*}
Applying inequality \eqref{steiner-measure} and using the fact that $P(s)\le P(\Omega)/2$ in the interval $(\Bar{t}, \frac{\abs{\Omega}}{P(\Omega)})$, we get
\begin{equation*}
    \mu\left(\frac{\abs{\Omega}}{P(\Omega)}\right)\ge \abs{\Omega}-P(\Omega) \overline{t}- \frac{P(\Omega)}{2}\left(\frac{\abs{\Omega}}{P(\Omega)}-\overline{t}\right)=\frac{\abs{\Omega}}{2}-\frac{P(\Omega)}{2}\overline{t}.
\end{equation*}
Previous inequality, combined with \eqref{absurd}, gives

$$\frac{\abs{\Omega}}{6n}\left(\frac{P(\Omega)R_\Omega}{|\Omega|}-1\right)> \frac{\abs{\Omega}}{2}-\frac{P(\Omega)}{2}\overline{t},$$
and rearranging the terms we have the following bounds on $\Bar{t}$
$$\frac{\abs{\Omega}}{P(\Omega)}\ge \overline{t}> \frac{\abs{\Omega}}{P(\Omega)}\left(1-\frac{2}{6n}\left(\frac{P(\Omega)R_\Omega}{\abs{\Omega}}-1\right)\right).$$
Now, since $\mu(0)=\abs{\Omega}$ and \eqref{absurd} holds, then the convexity  of $\mu(\cdot)$ ensures that $\mu(t)$ is under the line passing through the points $(0,\abs{\Omega})$ and $(\frac{\abs{\Omega}}{P(\Omega)},q_1(n,\Omega)\abs{\Omega})$, for all $t \in [0,\frac{\abs{\Omega}}{P(\Omega)}]$, that is 
\begin{equation*}
    \mu(t) \le \abs{\Omega}+(q_1(n,\Omega)-1)P(\Omega)t, \qquad t\in \bigg[0,\frac{\abs{\Omega}}{P(\Omega)}\bigg].
\end{equation*}In particular, since $\Bar{t}\in [0,\frac{\abs{\Omega}}{P(\Omega)}]$ and using the fact that $q_1\equiv q_1(n,\Omega)\le 1$, we arrive to
$$\mu(\overline{t})\le \abs{\Omega}+ (q_1-1)P(\Omega)\overline{t}<\abs{\Omega}+ (q_1-1)|\Omega|(1-2q_1)\le 3q_1 \abs{\Omega}. $$
On the other hand, since $P(\overline{t}) \geq \frac{P(\Omega)}{2}$, we have
\begin{equation*}
    R_{\Omega} - \frac{\abs{\Omega}}{P(\Omega)} \leq  R_{\Omega} - \overline{t} \leq n \frac{\mu(\overline{t})}{P(\overline{t})} < n \frac{6q_1\abs{\Omega}}{P(\Omega)},
\end{equation*}
that can be rewritten as follows
$$
 \frac{P(\Omega) R_\Omega}{\abs{\Omega}}-1 < 6n \frac{1}{6n} \left(\frac{P(\Omega) R_\Omega}{\abs{\Omega}}-1\right),
$$
that is a contraddiction.\\
Let us now prove inequality \eqref{eq:LemmaP(MP)}, using the monotonicity of the perimeter we know that
\begin{equation}\label{eq:estimatemeasure}
    \mu(t)=\abs{\Omega}-\int_{0}^{t}P(s)\,ds\le \abs{\Omega}-P(t)t. 
\end{equation}
Moreover, using the upper bound in Proposition \ref{prop:PRM-1} for the superlevel set $\Omega_t$, we have that
\begin{equation}\label{eq:PRM_t}
    \frac{P(t)(R_\Omega-t)}{n}\le \mu(t).
\end{equation}
So that using \eqref{eq:estimatemeasure} and \eqref{eq:PRM_t}, evaluated at $t=\abs{\Omega}/P(\Omega)$, we get
\begin{equation*}
    \frac{1}{n}P\left(\frac{\abs{\Omega}}{P(\Omega)}\right)\left(R_\Omega-\frac{\abs{\Omega}}{P(\Omega)}\right)\le \abs{\Omega}-P\left(\frac{\abs{\Omega}}{P(\Omega)}\right)\frac{\abs{\Omega}}{P(\Omega)},
\end{equation*}
that gives
\begin{equation*}
   P\left(\frac{\abs{\Omega}}{P(\Omega)}\right)\left[\frac{1}{n}\left(\frac{P(\Omega)R_\Omega}{\abs{\Omega}}-1\right)+1\right] \le P(\Omega),
\end{equation*}
and eventually \eqref{eq:LemmaP(MP)}.
\end{proof}

We are now in position to prove Theorem \ref{thm:quantitative_PRM}.

\begin{proof}[Proof of Theorem \ref{thm:quantitative_PRM}]Let start from the lower bound.
 Using Proposition \ref{prop:estimatestorsion}, we have
  \begin{align*}
     T(\Omega)&\geq \int_{0}^{R_\Omega}\dfrac{\mu^2(t)}{P(t)}dt\\
      &\geq \dfrac{1}{P^2(\Omega)}\int_0^{\frac{|\Omega|}{P(\Omega)}}\left(|\Omega|-P(\Omega)t\right)^2 P(\Omega) dt
      +\dfrac{1}{P^2(\Omega)}\int_{\frac{|\Omega|}{P(\Omega)}}^{R_\Omega} \mu^2(t)(-\mu'(t))dt \\
      & =\dfrac{|\Omega|^3}{3 P^2(\Omega)}+\dfrac{1}{P^2(\Omega)}\cdot\dfrac{\mu^3\left(\dfrac{|\Omega|}{P(\Omega)}\right)}{3}.
  \end{align*}

 So that, applying Lemma \ref{lemma:M(PM)eP(MP)}, we obtain 
\begin{align*}
      \dfrac{T(\Omega)P^2(\Omega)}{|\Omega|^3}-\dfrac{1}{3}\geq
      \dfrac{1}{3}\cdot\dfrac{\mu^3\left(\dfrac{|\Omega|}{P(\Omega)}\right)}{|\Omega|^3}\geq \dfrac{q_1(n,\Omega)^3}{3}=\dfrac{1}{2^3\cdot3^4  n^3}\left(\frac{P(\Omega) R_\Omega}{\abs{\Omega}}-1\right)^3,
  \end{align*}
   and this proves the lower bound in \eqref{eq:TPM-PRM}.

\noindent For what it concerns the upper bound, we use the Makai inequality
  $$
  T(\Omega) \leq \frac{1}{3} R_\Omega^2 \abs{\Omega}.$$
Then
  $$
  \dfrac{T(\Omega)P^2(\Omega)}{|\Omega|^3}-\frac{1}{3} \leq \frac{1}{3}\left(\frac{P^2(\Omega) R^2_\Omega}{\abs{\Omega}^2}-1\right)=\frac{1}{3}\left(\frac{P(\Omega) R_\Omega}{\abs{\Omega}}+1\right)\left(\frac{P(\Omega) R_\Omega}{\abs{\Omega}}-1\right)\leq\frac{n+1}{3}\left(\frac{P(\Omega) R_\Omega}{\abs{\Omega}}-1\right).$$
  For the optimality of the exponent of the upper bound see Proposition \ref{prop:sharpnessexponent}.
  
\end{proof}
For the P\'olya eigenvalue functional, we prove Theorem \ref{thm:quantitative_lambdapr}.
\begin{proof}[Proof of Theorem \ref{thm:quantitative_lambdapr}]
 The first lines of the proof follow the same argument proposed in \cite{polya1960}, whose computations are analogous to the one shown in Subsection \ref{subsec:webtors}. Let us use as a test function in the variational characterization of $\lambda(\Omega)$ the function $f(x)=g(t)$, where $g$ depends only on the distance function from the boundary of $\Omega$. Then by coarea formula we get

\begin{equation}
\label{polyalambda}
    \lambda(\Omega) \leq \dfrac{\displaystyle{\int_0^{R_\Omega} (g'(t))^2 P(t) \, dt}}{\displaystyle{\int_0^{R_\Omega} g^2(t) P(t) \, dt}}.
\end{equation}
The latest, with the change of  variables $s= \frac{\pi}{2}\frac{\mu(t)}{\abs{\Omega}}$, leads to 

\begin{equation*}
    \lambda(\Omega)\le\frac{\pi^2}{4\abs{\Omega}^2}\frac{\displaystyle{\int_{0}^{\frac{\pi}{2}}h'(s)^2 P(t)^2\,ds}}{\displaystyle{\int_0^{\frac{\pi}{2}}h(s)^2\,ds}},
\end{equation*}
where $h(s)=g(t)$, with $h(\frac{\pi}{2})=0$. Now, we choose $\Bar{t}= \frac{\abs{\Omega}}{P(\Omega)}$ and we  denote by 
\begin{equation}\label{eq:bars}
    \Bar{s}= \frac{\pi}{2}\frac{\mu(\Bar{t})}{\abs{\Omega}}.
    \end{equation}
Hence, we divide the integral at the numerator in \eqref{polyalambda} at $\Bar{s}$, obtaining

\begin{equation*}
\begin{split}
    \lambda(\Omega) &\le \frac{\pi^2}{4\abs{\Omega}^2}\cdot\frac{\displaystyle \int_{0}^{\Bar{s}}h'(s)^2P(t)^2\,ds +\int_{\Bar{s}}^{\frac{\pi}{2}}h'(s)^2P(t)^2\,ds}{\displaystyle\int_0^{\frac{\pi}{2}}h(s)^2\,ds}\\
    &\le\frac{\pi^2}{4\abs{\Omega}^2}\cdot\frac{\displaystyle P(\Omega)\int_{0}^{\Bar{s}}h'(s)^2P(t)\,ds +P(\Omega)^2\int_{\Bar{s}}^{\frac{\pi}{2}}h'(s)^2\,ds}{\displaystyle\int_0^{\frac{\pi}{2}}h(s)^2\,ds}.
    \end{split}
\end{equation*}
Using the monotonicity of the perimeter in the first integral, we have that $P(t)\le P(\abs{\Omega}/P(\Omega))$ and applying Lemma \ref{lemma:M(PM)eP(MP)}, we have
\begin{equation}\label{eq:estimateintermediate}
\begin{split}
    \lambda(\Omega) &\le\frac{\pi^2}{4\abs{\Omega}^2}\cdot\frac{\displaystyle q_2(n,\Omega)P(\Omega)^2\int_{0}^{\Bar{s}}h'(s)^2\,ds +P(\Omega)^2\int_{\Bar{s}}^{\frac{\pi}{2}}h'(s)^2\,ds}{\displaystyle\int_0^{\frac{\pi}{2}}h(s)^2\,ds}\\
    &=\frac{\pi^2P(\Omega)^2}{4\abs{\Omega}^2}\cdot\frac{\displaystyle (q_2(n,\Omega)-1)\int_{0}^{\Bar{s}}h'(s)^2\,ds +\int_{0}^{\frac{\pi}{2}}h'(s)^2\,ds}{\displaystyle\int_0^{\frac{\pi}{2}}h(s)^2\,ds}.
\end{split}
\end{equation}
Now we choose $h(s)=\cos(s)$, so that
\begin{equation*}
    \int_0^\frac{\pi}{2}h'(s)^2\,ds=\int_0^\frac{\pi}{2}h(s)^2\,ds= \frac{\pi}{4}.
\end{equation*}
In this way, multiplying \eqref{eq:estimateintermediate} by $\abs{\Omega}^2/P(\Omega)^2$, we get
\begin{equation*}
    \frac{\pi^2}{4}-\frac{\lambda(\Omega)\abs{\Omega}^2}{P(\Omega)^2}\ge\pi (1-q_2(n,\Omega))\int_0^{\Bar{s}}\sin^2(s)\,ds. 
\end{equation*}
Using the inequality $\sin(s)\ge \frac{2}{\pi}s$, which is valid for every $s\in [0,\pi/2]$, then 
\begin{equation}\label{eq:estimatebars3}
    \frac{\pi^2}{4}-\frac{\lambda(\Omega)\abs{\Omega}^2}{P(\Omega)^2}\ge\frac{4}{3\pi}(1-q_2(n,\Omega))\Bar{s}^3.
\end{equation}
Recalling \eqref{eq:bars} and Lemma \ref{lemma:M(PM)eP(MP)}, we have that
\begin{equation}\label{eq:estimatebars}
    \Bar{s}^3=\frac{\pi^3}{8}\cdot\frac{\mu\left(\frac{\abs{\Omega}}{P(\Omega)}\right)^3}{\abs{\Omega}^3}\ge\frac{\pi^3}{8\cdot 6^3 n^3}\left(\frac{P(\Omega)R_\Omega}{|\Omega|}-1\right)^3.
\end{equation}
Moreover, by the definition of $q_2(n,\Omega)$ and equation \eqref{convex_estimates}, we get
\begin{equation}\label{eq:estimateintermediate2}
    1-q_2(n,\Omega)= 1-\frac{1}{1+\frac{1}{n}\left(\frac{P(\Omega)R_\Omega}{\abs{\Omega}}-1\right)}= \frac{\frac{1}{n}\left(\frac{P(\Omega)R_\Omega}{\abs{\Omega}}-1\right)}{1+\frac{1}{n}\left(\frac{P(\Omega)R_\Omega}{\abs{\Omega}}-1\right)}\ge \frac{1}{2n-1}\left(\frac{P(\Omega)R_\Omega}{\abs{\Omega}}-1\right).
\end{equation}
Putting  \eqref{eq:estimatebars} and \eqref{eq:estimateintermediate2} in \eqref{eq:estimatebars3}, we have
\begin{equation*}
      \frac{\pi^2}{4}-\frac{\lambda(\Omega)\abs{\Omega}^2}{P(\Omega)^2}\ge\frac{\pi^2}{2^5\cdot 3^4}\cdot\frac{1}{n^3(2n-1)}\left(\frac{P(\Omega)R_\Omega}{\abs{\Omega}}-1\right)^4,
\end{equation*}
which concludes the proof of the lower bound.\\
Regarding the upper bound, the proof is a direct consequence of the Hersch-Protter inequality.
Indeed, using \eqref{eq:Hersch-Protter} and \eqref{convex_estimates}, we get
\begin{equation*}
    \frac{\pi^2}{4} -\frac{\lambda(\Omega)\abs{\Omega}^2}{P^2(\Omega)}  \le \frac{\pi^2}{4}\bigg(1-\frac{\abs{\Omega}^2}{P^2(\Omega)R_\Omega^2}\bigg)=\frac{\pi^2}{4}\bigg(1+\frac{\abs{\Omega}}{P(\Omega)R_\Omega}\bigg)\bigg(1-\frac{\abs{\Omega}}{P(\Omega)R_\Omega}\bigg)\le \frac{\pi^2}{2}\bigg(\frac{P(\Omega)R_\Omega}{\abs{\Omega}}-1\bigg).
\end{equation*}
For the optimality of the exponent of the upper bound see Proposition \ref{prop:sharpnessexponent}.
\end{proof}

Eventually we give the proof of Theorem \ref{thm:quantitative_torsionmakai_PRM}.

\begin{proof}[Proof of Theorem \ref{thm:quantitative_torsionmakai_PRM}]
    Let us start by the lower bound. In this case we use the upper bound \eqref{eq:lowboundpolya}
    \begin{equation*}
        T(\Omega) \le \int_{\Omega}d(x,\partial\Omega)^2\,dx.
    \end{equation*}
    Applying Coarea Formula, integrating by parts and using estimate \eqref{measure_estimate1}, we get
    \begin{equation}\label{eq:makaieq1}
        T(\Omega) \le \int_0^{R_\Omega}t^2P(t)\,dt= 2\int_0^{R_\Omega}t\mu(t)\,dt\le 2\int_0^{R_\Omega}t(R_\Omega-t)P(t)\,dt+ \frac{1}{n-1}\int_0^{R_\Omega}t(R_\Omega-t)^2P'(t)\,dt.
    \end{equation}
    If we integrate by parts the second integral on the right-hand side of \eqref{eq:makaieq1}, we get
    \begin{equation}\label{eq:makaieq2}
    \begin{split}
\int_0^{R_\Omega}t(R_\Omega-t)^2P'(t)\,dt&= t(R_\Omega-t)^2P(t)\bigg|_0^{R_\Omega}-\int_0^{R_\Omega}[(R_\Omega-t)^2-2t(R_\Omega-t)]P(t)\,dt\\
&=2\int_0^{R_\Omega}t(R_\Omega-t)P(t)\,dt-\int_0^{R_\Omega}(R_\Omega-t)^2P(t)\,dt,
    \end{split},
    \end{equation}
    where we notice that one of the two integrals in \eqref{eq:makaieq2} is equal to the one in \eqref{eq:makaieq1}. Therefore
    \begin{equation*}
        \begin{split}\int_0^{R_\Omega}t^2P(t)\,dt&\le \frac{2n}{n-1}\int_0^{R_\Omega}t(R_\Omega-t)P(t)\,dt-\frac{1}{n-1}\int_0^{R_\Omega}(R_\Omega-t)^2P(t)\,dt \\
        &=2\frac{n+1}{n-1}R_\Omega\int_0^{R_\Omega}tP(t)\,dt-\frac{2n+1}{n-1}\int_0^{R_\Omega}t^2P(t)\,dt-\frac{R_\Omega^2\abs{\Omega}}{n-1}.
        \end{split}
    \end{equation*}
    Summing up the same terms, we have
    \begin{equation}\label{eq:makaieq3}
        \int_0^{R_\Omega}t^2P(t)\,dt\le \frac{2(n+1)}{3n}R_\Omega \int_{0}^{R_\Omega}tP(t)\,dt-\frac{R_\Omega^2\abs{\Omega}}{3n}.
    \end{equation}
    We now estimate the integral on the right-hand side of \eqref{eq:makaieq3}. Integrating by parts and using again \eqref{eq:dermu} and \eqref{measure_estimate1}
    \begin{equation*}
    \begin{split}
\int_{0}^{R_\Omega}tP(t)\,dt= \int_{0}^{R_\Omega}\mu(t)\,dt&\le \int_{0}^{R_\Omega}(R_\Omega-t)P(t)\,dt+ \frac{1}{2(n-1)}\int_{0}^{R_\Omega}(R_\Omega-t)^2P'(t)\,dt\\
&=\frac{n}{n-1}\int_{0}^{R_\Omega}(R_\Omega-t)P(t)\,dt-\frac{R_\Omega^2P(\Omega)}{2(n-1)}\\
&=\frac{n}{n-1}R_\Omega\abs{\Omega}-\frac{n}{n-1}\int_{0}^{R_\Omega}tP(t)\,dt-\frac{R_\Omega^2P(\Omega)}{2(n-1)}.
\end{split}
    \end{equation*}
    Therefore
    \begin{equation}\label{eq:makaieq4}
        \int_{0}^{R_\Omega}tP(t)\,dt\le \frac{n}{2n-1}R_\Omega\abs{\Omega}-\frac{R_\Omega^2P(\Omega)}{2(2n-1)}.
    \end{equation}
    Inserting \eqref{eq:makaieq4} into \eqref{eq:makaieq3}, we get
    \begin{equation*}
       \begin{split}
       \int_0^{R_\Omega}t^2P(t)\,dt&\le \frac{2(n+1)}{3(2n-1)}R_\Omega^2\abs{\Omega}+\frac{n+1}{3n(2n-1)} R_\Omega^3P(\Omega)-\frac{R_\Omega^2\abs{\Omega}}{3n}\\
       &=\frac{R_\Omega^2\abs{\Omega}}{3}+\frac{n+1}{3n(2n-1)}\bigg[R_\Omega^2\abs{\Omega}-R_\Omega^3P(\Omega)\bigg]\\
       &=\frac{R_\Omega^2\abs{\Omega}}{3}+\frac{n+1}{3n(2n-1)}R_\Omega^2\abs{\Omega}\beta(\Omega).
\end{split}
\end{equation*}
    Considering \eqref{eq:makaieq1}, dividing by $R^2_\Omega\abs{\Omega}$, we arrive to the conclusion
    \begin{equation*}
        \frac{1}{3}-\frac{T(\Omega)}{R^2_\Omega\abs{\Omega}}\ge \frac{n+1}{3n(2n-1)}\beta(\Omega).
    \end{equation*}
    Let us now prove the upper bound. If we multiply and divide the functional by $P^2(\Omega)/\abs{\Omega}$ and use the lower bound for the P\'olya functional, we get
\begin{equation*}
\begin{split}
\frac{1}{3}-\frac{T(\Omega)}{R_\Omega^2 \abs{\Omega}}&= \frac{1}{3}-\frac{T(\Omega)P^2(\Omega)}{ \abs{\Omega}^2}\cdot\frac{\abs{\Omega}^2}{R_\Omega^2 P^2(\Omega)}\le \frac{1}{3}\bigg(1-\frac{\abs{\Omega}^2}{R_\Omega^2 P^2(\Omega)}\bigg)\\
    &= \frac{1}{3}\bigg(1+\frac{\abs{\Omega}}{R_\Omega P(\Omega)}\bigg)\bigg(1-\frac{\abs{\Omega}}{R_\Omega P(\Omega)}\bigg)\le \frac{2}{3}\bigg(\frac{P(\Omega)R_\Omega}{\abs{\Omega}}-1\bigg).
        \end{split}
\end{equation*}
    \end{proof}

\section{Corollaries and other results}Let us prove Proposition \ref{prop:quantitative_wd}.

 \begin{proof}[Proof of Proposition \ref{prop:quantitative_wd}]

Let us start with the P\'olya torsion functional.
 The lower bound in \eqref{eq:lowboundpolya} leads to 
\begin{equation*}
    T(\Omega) \geq \int_0^{R_\Omega}\frac{\mu^2(t)}{P(t)} \, dt.
\end{equation*}
At this point, let us split the integral above at the value $\overline{t}$ defined for some $\Tilde{c}\in (0,1)$ as
\begin{equation}
    \label{labelcomega}
    \mu(\overline{t})= \Tilde{c}\abs{\Omega}.
\end{equation}
It certainly exists, since the distance function is a $W^{1,\infty}$ function with gradient different from 0 a.e.\,.

Hence, by using \eqref{perimeter_estimate} and \eqref{labelcomega}, we write
\begin{align*}
    T(\Omega)&\geq \int_0^{\overline{t}}\frac{\mu^2(t)}{P(t)} \, dt+\int_{\overline{t}}^{R_\Omega}\frac{\mu^2(t)}{P(t)} \, dt \\
    &\geq  \frac{1}{P^2(\Omega)} \int_0^{\overline{t}}\mu^2(t)(-\mu'(t))\, dt+ \frac{1}{\displaystyle{P(\Omega)\left(P(\Omega) - c_n \frac{\abs{\Omega} - \mu(\overline{t})}{P^{\frac{1}{n-1}}(\Omega)}\right)}}\int_{\overline{t}}^{R_\Omega}\mu^2(t) (-\mu'(t))\, dt\\
    &= \frac{1}{P^2(\Omega)} \cdot\frac{\abs{\Omega}^3-\mu^3(\overline{t}) }{3}+ \frac{1}{\displaystyle{P(\Omega)\left(P(\Omega) - c_n \frac{\abs{\Omega} - \mu(\overline{t})}{P^{\frac{1}{n-1}}(\Omega)}\right)}} \frac{\mu^3(\overline{t}) }{3}\\
    & \geq \frac{1}{P^2(\Omega)} \cdot\frac{\abs{\Omega}^3-\mu^3(\overline{t}) }{3}+ \frac{1}{P^2(\Omega)}\left(1 +c_n\frac{\abs{\Omega} - \mu(\overline{t})}{P^{\frac{n}{n-1}}(\Omega)}\right) \frac{\mu^3(\overline{t}) }{3}\\
    &= \frac{\abs{\Omega}^3}{3P^2(\Omega)}+  c_n\frac{(1-\Tilde{c})\Tilde{c}^3}{P^{2+\frac{n}{n-1}}(\Omega)}\cdot\frac{\abs{\Omega}^4 }{3}.
\end{align*}
Now we choose $\Tilde{c}$ in order to maximize $(1-\Tilde{c})\Tilde{c}^3$. So we find the maximum in $(0,1)$ of the function $f(x)=(1-x)x^3$, which gives
\begin{equation*}
    \Tilde{c}= \frac{3}{4}.
\end{equation*}
Hence, we have
\begin{align}\label{eq:estimate1}
   \frac{ T(\Omega)P^2(\Omega)}{\abs{\Omega}^3}&\geq \frac{1}{3}+ \frac{27 c_n}{256} \cdot\frac{\abs{\Omega}}{P(\Omega)}\cdot\frac{1 }{P^{\frac{1}{n-1}}(\Omega)} \geq \frac{1}{3}+ \frac{27 c_n}{256n}\cdot\frac{R_\Omega}{P^{\frac{1}{n-1}}(\Omega)}.
\end{align}

\noindent Combining \eqref{eq:estimate1} with \eqref{eq:lowboundinradius} and \eqref{eq:perdiam}, we get the assertion with

\begin{equation*}
    \displaystyle C_1(n)=\begin{cases}
     \displaystyle{\frac{27}{256n^{\frac{3n-2}{2(n-1)}}}\sqrt{\frac{n+2}{2n+2}}} & n \,\, \text{even}\\ \\
      \displaystyle{\frac{27}{256n^{\frac{3n-2}{2(n-1)}}}\sqrt{\frac{n+2}{4n}}} & n \,\, \text{odd}.
    \end{cases}
\end{equation*}
\hspace{5mm}

For what it concerns the P\'olya eigenvalue functional, we start from \eqref{polyalambda}.
At this point, let us split the integral above at the value $\overline{t}$ defined as
$$
\mu(\overline{t})= \frac{\abs{\Omega}}{2}.
$$
It certainly exists, since the distance function is a $W^{1,\infty}$ function with gradient different from 0 a.e.. Hence we write
\begin{equation*}
    \lambda(\Omega) \leq \dfrac{\displaystyle{\int_0^{\overline{t}} (g'(t))^2 P(t) \, dt+\int_{\overline{t}}^{R_\Omega} (g'(t))^2 P(t) \, dt}}{\displaystyle{\int_0^{R_\Omega} g^2(t) P(t) \, dt}}.
\end{equation*}
Now, performing the same change of variable proposed by P\'olya
\begin{equation*}
    s= \frac{\pi\mu(t)}{2|\Omega|},
\end{equation*}
we get
\begin{equation}\label{eq:lambdachangevar}
    \lambda(\Omega) \le \frac{\pi^2}{4\abs{\Omega}^2}\frac{\displaystyle \int_{\frac{\pi}{4}}^{\frac{\pi}{2}}h'(s)^2P(t)^2\,ds +\int_0^{\frac{\pi}{4}}h'(s)^2P(t)^2\,ds}{\displaystyle\int_0^{\frac{\pi}{2}}h(s)^2\,ds},
\end{equation}
where $h(s)=g(t)$ and $h(\frac{\pi}{2})=0$. Since for each $s \in \left[0, \frac{\pi}{4}\right]$ we have $t\in [\overline{t}, R]$, in such interval we have by \eqref{perimeter_estimate}
$$
 P(t)\leq P(\Omega) - c_n \frac{\abs{\Omega} - \mu(t)}{P^{\frac{1}{n-1}}(\Omega)} \leq P(\Omega) - c_n \frac{\abs{\Omega} - \mu(\overline{t})}{P^{\frac{1}{n-1}}(\Omega)}=P(\Omega) - \frac{c_n}{2}\cdot\frac{\abs{\Omega}}{P^{\frac{1}{n-1}}(\Omega)}.
$$
Hence, \eqref{eq:lambdachangevar} gives
\begin{align*}\label{eq:bho}
    \lambda(\Omega) &\le \frac{\pi^2}{4\abs{\Omega}^2}\cdot\frac{\displaystyle P(\Omega)^2\int_{\frac{\pi}{4}}^{\frac{\pi}{2}}h'(s)^2\,ds +P(\Omega) \left(P(\Omega) - \frac{c_n}{2}\cdot\frac{\abs{\Omega}}{P^{\frac{1}{n-1}}(\Omega)}\right)\int_0^{\frac{\pi}{4}}h'(s)^2\,ds}{\displaystyle\int_0^{\frac{\pi}{2}}h(s)^2\,ds},\\
     &=\frac{\pi^2P(\Omega)^2}{4\abs{\Omega}^2}  \left(\frac{\displaystyle \int_{0}^{\frac{\pi}{2}}h'(s)^2\,ds }{\displaystyle\int_0^{\frac{\pi}{2}}h(s)^2\,ds}-\frac{c_n}{2}\frac{\abs{\Omega}}{P(\Omega)}\cdot\frac{1}{P^{\frac{1}{n-1}}(\Omega)}\cdot\frac{ \displaystyle\int_0^{\frac{\pi}{4}}h'(s)^2\,ds}{\displaystyle\int_0^{\frac{\pi}{2}}h(s)^2\,ds}\right).
\end{align*}
Choosing $h(t)= cos(t)$, we get 
\begin{equation}\label{eq:const}
    \frac{\displaystyle\int_0^{\frac{\pi}{2}}h'(s)^2\,ds}{\displaystyle\int_0^{\frac{\pi}{2}}h(s)^2\,ds}=1\qquad\frac{\displaystyle\int_0^{\frac{\pi}{4}}h'(s)^2\,ds}{\displaystyle\int_0^{\frac{\pi}{2}}h(s)^2\,ds}=\frac{\pi -2}{8}\cdot\frac{4}{\pi}=\frac{\pi -2}{2\pi}.
\end{equation}
Then equation \eqref{eq:const} and $P(\Omega)R \leq \abs{\Omega} n$, gives 
\begin{equation}
    \label{quasidef}
    \frac{\lambda(\Omega)\abs{\Omega}^2}{P(\Omega)^2}\leq\frac{\pi^2}{4} -  \frac{\pi^2 c_n}{8n} \left(\frac{\pi-2}{2\pi}\right) \frac{R_\Omega}{P^{\frac{1}{n-1}}(\Omega)}.
\end{equation}
Again, combining \eqref{quasidef} with \eqref{eq:lowboundinradius} and \eqref{eq:perdiam}, we get the assertion with 
\begin{equation*}
    \displaystyle C_2(n)=\begin{cases}
     \displaystyle{\frac{\pi^2(n-1)}{8n^{\frac{n}{n-1}}}\cdot\frac{(\pi-2)}{2\pi}\sqrt{\frac{n+2}{n(2n+2)}}} & n \,\, \text{even}\\ \\
       \displaystyle{\frac{\pi^2(n-1)}{8n^{\frac{n}{n-1}}}\cdot\frac{(\pi-2)}{2\pi}\cdot\frac{\sqrt{n+2}}{2n}} & n \,\, \text{odd}.
    \end{cases}
\end{equation*}
About the sharpness of the exponent in both the inequalities \eqref{quantitative_width} and \eqref{quantitative_width2}, see Proposition \ref{prop:sharpnessexponent}, while for the counterexample see Remark \ref{rem:counterexamples}.
\end{proof}

Finally we give the proof of Corollary \ref{cor:quantitative_lr}.

\begin{proof}[Proof of Corollary \ref{cor:quantitative_lr}]
If we multiply and divide the functional by $\abs{\Omega}^2/P(\Omega)^2$ and use the lower bound for the P\'olya functional, we get
\begin{equation*}
\begin{split}
\lambda(\Omega)R_\Omega^2 - \frac{\pi^2}{4}&= \frac{\lambda(\Omega)\abs{\Omega}^2}{ P(\Omega)^2} \cdot\frac{P(\Omega)^2R_\Omega^2 }{\abs{\Omega}^2} - \frac{\pi^2}{4}\le \frac{\pi^2}{4}\bigg(\frac{P(\Omega)^2R_\Omega^2 }{\abs{\Omega}^2}-1\bigg)\\
    &= \frac{\pi^2}{4}\bigg(\frac{P(\Omega)R_\Omega }{\abs{\Omega}}+1\bigg)\bigg(\frac{P(\Omega)R_\Omega }{\abs{\Omega}}-1\bigg)\le \frac{\pi^2(n+1)}{4}\bigg(\frac{P(\Omega)R_\Omega}{\abs{\Omega}}-1\bigg).
        \end{split}
\end{equation*}
\end{proof}


\section{Sharpness and counterexamples}\label{Sec:sharpnessandcounterexamples}
In this Section, we prove the sharpness of the exponents of the lower bounds in \eqref{eq:PRM-WD},   
\eqref{quantitative_width} and  \eqref{quantitative_width2},  and the upper bounds in \eqref{eq:TPM-PRM} and \eqref{quantitative_HMP-PRM}.
To this aim, we will exhibit a suitable sequence of convex sets for which the deficit (i.e. the absolute value of the difference between the functional and the optimal constant)  divided by the remainder term converges to a positive constant.  


\begin{prop}\label{prop:sharpnessexponent}   Let $\Omega\in \mathcal{K}_n$. Then the exponents of the lower bounds in \eqref{eq:PRM-WD}, 
\eqref{quantitative_width} and  \eqref{quantitative_width2},  and of the upper bounds in \eqref{eq:TPM-PRM} and  \eqref{quantitative_HMP-PRM}
    are sharp.  
\end{prop}

\begin{proof}
Let us  consider the following family of thinning parallelepipeds 
$$
\Omega_a = [0,1]^{n-1} \times [0,a], \text{ with } a \to 0.
$$
By simple computations we can explicit the behavior of the quantities involved in our functionals  in terms of $a$ (see Table \ref{tab:tab1}).
\begin{table}[!ht]
\centering
\begin{tabular}{|c||c|c|c|c|c|c|c|}
    \hline

    & & & & & & &\\[-0.5ex]
   Set
    & $|\Omega|$ &  $P(\Omega)$  &  $R_\Omega$ & $w(\Omega)$ & $\diam(\Omega)$ & $T(\Omega)$ & $\lambda(\Omega)$ 
    \tabularnewline[3ex]
    \hline\hline

    & & & & & & &\\[-0.5ex]
    $\Omega_a$  & $a$ & $2 + \mathcal{H}^{n-2}\left(\partial [0,1]^{n-1}\right) a$ & $\frac{a}{2}$ & $a$ & $\sqrt{n-1+a^2}$ & $\simeq \frac{a^3}{12}$& $\pi^2 \left(n-1+\frac{1}{a^2}\right)$ 
    \tabularnewline[3ex]
    \hline

\end{tabular}
\caption{Values of the functional on $ \Omega_a$.}\label{tab:tab1}
\end{table}
Let us start with \eqref{eq:PRM-WD}, \eqref{quantitative_width} and  \eqref{quantitative_width2}, by considering $\Omega_a$ by Table \ref{tab:tab1} we get that there exist positive constants $Q_1$ and $Q_2$ such that
\begin{align*}
\frac{P(\Omega_a)R_{\Omega_a}}{\abs{\Omega_a}}  -1 &\leq Q_1 \,a  \leq Q_2 \frac{w_{\Omega_a}}{\diam(\Omega_a)}; \\
\frac{T(\Omega_a)P^2(\Omega_a)}{\abs{\Omega_a}^3}- \frac{1}{3} &\leq Q_1\, a \leq Q_2 \frac{w_{\Omega_a}}{\diam(\Omega_a)};\\
\frac{\pi^2}{4}- \frac{\lambda(\Omega_a)\abs{\Omega_a}^2}{ P^2(\Omega_a)}&\leq Q_1\, a \leq Q_2  \frac{w_{\Omega_a}}{\diam(\Omega_a)}.
\end{align*}

 We now prove it for the upper bounds in \eqref{eq:TPM-PRM} and  \eqref{quantitative_HMP-PRM}. 
 From Table \ref{tab:tab1} there exists positive constants $Q_3$ and $Q_4$ such that
\begin{align*}
\frac{T(\Omega_a)P^2(\Omega_a)}{\abs{\Omega_a}^3}- \frac{1}{3} &\geq Q_3\, a \geq Q_4 \left(\frac{P(\Omega_a)R_{\Omega_a}}{\abs{\Omega_a}}  -1\right);\\
\frac{\pi^2}{4}- \frac{\lambda(\Omega_a)\abs{\Omega_a}^2}{ P^2(\Omega_a)}
&\geq Q_3\, a \geq Q_4 \left(\frac{P(\Omega_a)R_{\Omega_a}}{\abs{\Omega_a}}  -1\right). \\
\end{align*}
\end{proof}
\begin{oss}\label{rem:counterexamples}
    In what follows we give counterexamples of the non-validity of the reverse inequalities of Propositions 
\ref{prop:comparisonasymmetries}, \ref{prop:quantitative_wd} and Corollary  \ref{cor:quantitative_lr}. 

To do that we will consider the  the family of collapsing pyramids $C_\alpha$. Let us consider the $(n-1)$-dimensional square $Q_{n-1}=(-1/2,1/2)^{n-1}$ and a point $V\in \mathbb R^n$, such that $V=(0,...,\alpha)$, with $\alpha>0$. Then the family of collapsing pyramids is given by
\begin{equation*}
    C_\alpha=\text{convexhull}(Q_{n-1},V).
\end{equation*}

In particular, we resume the asymptotics of some geometric quantities for the family of pyramids as $\alpha$ go to 0:
\begin{equation}
    \label{asymp}
    \lim_{\alpha \to 0}\frac{\abs{C_\alpha}}{\alpha}=\frac 1 n, \quad \lim_{\alpha \to 0}P({C_\alpha})=2, \quad  \lim_{\alpha \to 0}\frac{R_{C_\alpha}}{\alpha}=\frac 1 2, \quad  \lim_{\alpha \to 0}\frac{w({C_\alpha})}{\alpha}=1, \quad \diam(C_\alpha)= \sqrt{2}.
\end{equation}

\begin{enumerate}
    \item  The asymptotics in \eqref{asymp} easily gives

\begin{equation*}
         \lim_{\alpha\to 0}\frac{w(C_\alpha)}{\diam(C_\alpha)}=0,
    \end{equation*}
while
\begin{equation}\label{eq:betapyramid}
    \lim_{\alpha\to 0}\frac{P(C_\alpha)R_{C_\alpha}}{\abs{C_\alpha}}-1 = n-1.
\end{equation}
This proves that a reverse inequality in Proposition \ref{prop:comparisonasymmetries} is not allowed.

 \item About Proposition \ref{prop:quantitative_wd} is not necessary to exhibit an explicit counterexamples, since it follows from the counterexample given in Proposition \ref{prop:comparisonasymmetries}. Indeed if there existed a set $\tilde{\Omega}$ for which 
 \begin{equation*}
     \tilde{c}_1(n)\alpha(\tilde{\Omega})\ge \frac{T(\tilde\Omega)P^2(\tilde\Omega)}{|\tilde\Omega|^3}-\frac{1}{3}, 
 \end{equation*}
  for some positive dimensional constant $\tilde c(n)$, then by Theorem  \ref{thm:quantitative_PRM} we would have
  \begin{equation*}
       \alpha(\tilde\Omega)\ge \tilde k(n)\beta(\tilde\Omega)^3,
  \end{equation*}
  for some positive dimensional constant $\tilde k(n)$. But this is not possible since it would go against Proposition \ref{prop:comparisonasymmetries}. For the same reason inequalities as
  \begin{equation*}
      \tilde c_2(n)\alpha(\tilde \Omega) \ge \frac{\pi^2}{4}-\frac{\lambda(\tilde\Omega)|\tilde\Omega|^2}{P^2(\tilde\Omega)}\qquad \text{or}\qquad \tilde c_3(n)\alpha(\tilde \Omega) \ge\frac{1}{3}-\frac{T(\tilde\Omega)}{R^2_{\tilde\Omega}|\tilde\Omega|}
  \end{equation*}
cannot hold.
 \item Taking into account \cite{BFNT2020}, we have
 $$\lim_{\alpha \to 0}\lambda(C_\alpha) R^2_{C_\alpha} =\frac \pi 4,$$
 since the projection on the minimizing supporting hyperplane is $Q_{n-1}$, its eigenvalue is fixed.

On the contrary, equation \eqref{eq:betapyramid} gives $\beta(C_\alpha) \to 0$ as $\alpha$ goes to 0.
This proves that a reverse inequality in Corollary \ref{cor:quantitative_lr} is not allowed.

\end{enumerate}

\end{oss}

\section{Conclusion and open problems}

To conclude, Theorems~\ref{thm:quantitative_PRM}, \ref{thm:quantitative_lambdapr}, and~\ref{thm:quantitative_torsionmakai_PRM} fully characterize the optimizing sequences of the P\'olya torsion functional, the P\'olya eigenvalue functional, and the Makai functional in terms of the optimizing sequences of the purely geometric functional $\beta(\Omega)$. 
We expect that a more precise description of the profiles of these optimizing sequences could be obtained starting from the explicit expression of $\beta$, but such a characterization is still missing.

We conclude by listing some open problems related to the present work.
\begin{itemize}
    \item As mentioned above, a first natural direction of research concerns the minimizing sequences of $\beta(\Omega)$. It seems reasonable to conjecture that such sequences resemble collapsing cylindroids, losing at least one dimension.
    
    \item As shown in \eqref{polyatorsion}, \eqref{polyaeigenvalue}, and \eqref{hershtype}, the functionals under consideration are bounded both from above and from below. In this work, we focused only on one of the two inequalities, namely the one for which the minimizing sequences behave like collapsing cylindroids.
    
    There are, however, other inequalities of interest, for instance
    \[
    \frac{\pi^2}{4n^2}
    < \frac{\lambda(\Omega)\,|\Omega|^2}{P(\Omega)^2}.
    \]
    In this case, the minimizing sequences are given by collapsing pyramids $\{C_\alpha \}$, as shown by Brasco in~\cite{brasco2018principal}. Since it is easy to check that $\beta(C_\alpha) $ goes to $ 3$ as the pyramids collapse, the functional $\beta$ cannot be used to characterize the minimizing sequences of this inequality. Since on sequences of collapsing pyramids $\{C_\alpha\}$ the functional $\frac{P(\Omega)R_\Omega}{\abs{\Omega}}$ achieves the value $n$, a possible direction in proving a quantitative result could be the following
    \begin{equation*}
        \frac{\pi^2}{4n^2}-\frac{\lambda(\Omega)\abs{\Omega}^2}{P(\Omega)^2}\ge \overline{c}(n) \bigg(n-\frac{P(\Omega)R_\Omega}{\abs{\Omega}}\bigg)^\gamma,
    \end{equation*}
    for some dimensional positive constant $\overline{c}$ and for some positive exponent $\gamma>0$.
\end{itemize}

\section*{Acknowledgements}
We thank the referees for their careful reading and constructive comments.

The first two authors were supported by the Project MUR PRIN-PNRR 2022:  "Linear and Nonlinear PDE’S: New directions and Applications", P2022YFA.\\
The third author was supported by the Project PRIN 2020 Prot. 2024NT8W4 "Nonlinear evolution PDEs, fluid dynamics and transport equations: theoretical foundations and applications", CUP C95F21010250001.\\
This work has been partially supported by GNAMPA group of INdAM.

\section*{Conflicts of interest and data availability statement}
The authors declare that there is no conflict of interest. Data sharing not applicable to this article as no datasets were generated or analyzed during the current study.

\bibliographystyle{plain}
\bibliography{biblio}

\Addresses 
\end{document}